\documentclass[a4paper,11pt]{article}

\usepackage{amsmath,empheq}
\usepackage{amssymb,amsthm}
\usepackage{graphicx}
\usepackage{color}
\usepackage{multicol}
\usepackage{enumerate}
\usepackage{subfig}
\usepackage{hyperref}
\usepackage{setspace}
\usepackage{authblk}
\usepackage[colorinlistoftodos]{todonotes}

\newcommand{\R}{\mathbb{R}}

\newcommand{\N}{\mathbb{N}}
\newcommand{\h}{\mathcal{H}}

\newcommand{\tw}{\mathcal{T}_\omega}

\newcommand{\ms}{\mathcal{S}}

\newcommand{\abr}[1]{\langle #1\rangle}
\newcommand{\vv}{\mathbf{v}}

\newtheorem{theorem}{Theorem}[section]

\newtheorem{lemma}[theorem]{Lemma}
\newtheorem{remark}[theorem]{Remark}

\numberwithin{equation}{section}

\newenvironment{altproof}[1]
{\noindent%\addvspace{0.3cm}
{\em Proof of {#1}}.}
{\nopagebreak\mbox{}\hfill $\Box$\par\addvspace{0.5cm}}

\newcommand{\ba}{\begin{aligned}}
\newcommand{\ea}{\end{aligned}}
\newcommand{\sumj}{\sum_{j=1}^n}
\newcommand{\sumk}{\sum_{k=1}^n}
\newcommand{\bfu}{\mathbf{u}}
\newcommand{\bfv}{\mathbf{v}}

\newcommand{\cP}{{\mathcal P}}

\newcommand{\cT}{{\mathcal T}}

\renewcommand{\dim}{{\rm dim}\,}
\newcommand{\al}{\alpha}
\newcommand{\be}{\beta}
\newcommand{\ga}{\gamma}

\newcommand{\eps}{\varepsilon}

\newcommand{\la}{\lambda}

\newcommand{\om}{\omega}
\newcommand{\De}{\Delta}

\newcommand{\La}{\Lambda}
\newcommand{\Om}{\Omega}
\newcommand{\pa}{\partial}

\title{Bifurcations for a Coupled Schr\"odinger System\\ with Multiple Components}
\author[1]{Thomas Bartsch}
\affil{\it\small Mathematisches Institut, University of Giessen, Arndtstr. 2, 35392 Giessen, Germany}
\author[2]{Rushun Tian}
\affil{\it\small Academy of Mathematics and System Science, Chinese Academy of Sciences, Beijing 100190, PR China}
\author[3]{Zhi-Qiang Wang}
\affil[3]{\it Department of Mathematics and Statistics, Utah State University, Logan, UT 84322, USA}
\date{}

\begin{document}

%%  Title
\maketitle
\pagestyle{myheadings}
\thispagestyle{empty}

\begin{abstract}
In this paper, we study local bifurcations of an indefinite elliptic system with multiple components:
\begin{equation*}
 \left\{\begin{array}{ll}
         -\De u_j + au_j = \mu_ju_j^3+\be\sum_{k\ne j}u_k^2u_j,\\
         u_j>0\ \ \hbox{in}\ \Om, u_j=0 \ \ \hbox{on}\ \partial\Om,\ j=1,\dots,n.
        \end{array}
 \right.
\end{equation*}
Here $\Om\subset\R^N$ is a smooth and bounded domain, $n\ge3$, $a<-\Lambda_1$ where $\Lambda_1$ is the principal eigenvalue of $(-\De, H_0^1(\Om))$; $\mu_j$ and $\be$ are real constants. Using the positive and non-degenerate solution of the scalar equation $-\De\om-\om=-\om^3$, $\om\in H_0^1(\Om)$, we construct a synchronized solution branch $\tw$. Then we find a sequence of local bifurcations with respect to $\tw$, and we find global bifurcation branches of partially synchronized solutions.
\end{abstract}

\section{Introduction}
In this paper, we study the bifurcations of solutions to the following elliptic system
\begin{equation}\label{equ:es-1}%{equ:elliptic-system}
 \left\{\begin{array}{ll}
         -\De u_j + a_ju_j = \mu_ju_j^3+\be\sum_{k\ne j}u_k^2u_j,\\
         u_j>0\ \ \hbox{in}\ \Om,\ u_j=0 \ \ \hbox{on}\ \partial\Om,\ j=1, \dots, n,
        \end{array}
 \right.
\end{equation}
where $\Om\subset\R^N$ is a smooth and bounded domain with $N\le3$. Let $\Lambda_1$ be the principal eigenvalue of $(-\De, H_0^1(\Om))$. We say \eqref{equ:es-1} is definite if $a_j>-\Lambda_1$ for all $j$ and indefinite if $a_j\le-\Lambda_1$ for at least one $j$, $1\le j\le n$. Without loss of generality, assume $\mu_1\le\mu_2\le\cdots\le\mu_n$. System \eqref{equ:es-1} is called a focusing system if $0<\mu_1\le\dots\le\mu_n$ and a defocusing system if $\mu_1\le\dots\le\mu_n<0$. For all the other possibilities of $\mu_j$, we call \eqref{equ:es-1} a mixed system.

System \eqref{equ:es-1} describes the standing wave solutions of coupled nonlinear Schr\"odinger systems, which have many applications in physics, see \cite{Esry-Greene-Burke-Bohn:1997, Mitchell-Chen-Shih-Segev:1996, Ruegg-Cavadini-etc:2003} for examples. Mathematically, extensive research has been done regarding, for instance, the existence and multiplicity of solutions to these systems. One can refer to \cite{Ambrosetti-Colorado:2006, Ambrosetti-Colorado:2007, Bartsch-Wang:2006, Bartsch-Wang-Wei:2007, Dancer-Wei-Weth:2010, Lin-Wei-CMP:2005, Lin-Wei-PDNP:2006, Liu-Wang:2008, Liu-Wang:2010, Maia-Montefusco-Pellacci:2006, Sirakov:2007, Tian-Wang:2011, Wei-Weth:2007} for various types of results using variational methods. A different approach based on bifurcation methods has been applied in \cite{Bartsch-Dancer-Wang:2010, Tian-Wang:2013-jan, Tian-Wang:2013-feb}. In \cite{Bartsch-Dancer-Wang:2010} the definite case of \eqref{equ:es-1} has been considered with $n=2$, $a_1=a_2$, and $0<\mu_1\le\mu_2$. There the authors first found a continuous branch of {\it synchronized} solutions in $(\be,u_1,u_2)\in\R\times H_0^1(\Om)\times H_0^1(\Om)$, that is with two linearly dependent components $u_j=\al_j\om$ being constant multiples of a single function $\om\in H_0^1(\Om)$ which is a solution to the scalar equation $-\De\om+\om=\om^3$. This solution branch exists for $\be\in(-\sqrt{\mu_1\mu_2}, \mu_1)\cup(\mu_2, \infty)$. Then the existence of infinitely many local bifurcations with respect to this branch has been obtained in \cite{Bartsch-Dancer-Wang:2010}. If $\Om$ is radially symmetric and $u_1, u_2$ are restricted to a radial function space, every local bifurcation gives rise to a global bifurcation branch in $\R\times H_0^1(\Om)\times H_0^1(\Om)$. In \cite{Tian-Wang:2013-jan, Tian-Wang:2013-feb}, the indefinite cases of \eqref{equ:es-1} were considered for $n=2$ and $\mu_1, \mu_2\in\R$. According to the values of $\mu_1$ and $\mu_2$, bifurcation diagrams were also obtained for $\be$ in certain intervals.

A natural generalization of the results in \cite{Bartsch-Dancer-Wang:2010, Tian-Wang:2013-jan, Tian-Wang:2013-feb} would be to extend them to a system that consists of more components, that is $n\ge3$. Note that the increase of the number of components brings new difficulties in analyzing the linearized system of \eqref{equ:es-1}, which is important in determining the bifurcation parameters and describing global bifurcations. When $n=2$, the linearized system can be reduced to one scalar equation that is related to a Sturm-Liouville type eigenvalue problem. Then the bifurcation parameters can be determined, and global bifurcations are obtained, since the kernel space of the linearized system generically only has dimension $1$. In the case $n\ge3$, these processes become more complicated and higher dimensional kernels appear due to the structure of the system.

Recently, the bifurcation of $n$-component systems has been investigated in \cite{Bartsch:2013} when $a_j\equiv a$ and $\mu_j>0$, that is in the focusing case. Similar to the two equation system, a synchronized solution branch exists (all components being synchronized), and a sequence of local bifurcations with respect to this branch was found. The structure of the system however forces the kernels of the linearization to be high-dimensional at the bifurcation points; more precisely, the dimensions are positive multiples of $n-1$, hence they can never be $1$ and are even if $n$ is odd. Using a hidden symmetry, the existence of global bifurcation branches was proved in \cite{Bartsch:2013}, consisting of solutions $(\be,u_1,\dots,u_n)\in\R\times\h$, $\h=[H_0^1(\Om)]^n$, where some but not all components are synchronized.

In this paper, we are interested in the bifurcation phenomena of solutions to \eqref{equ:es-1} when $n\ge3$ and with the additional symmetric requirement: $a_j\equiv a$ for $j=1,\dots,n$. Without of loss of generality we may assume $\La_1<1$ and take $a=-1$, thus we consider the system
\begin{equation}\label{equ:es-2}%{equ:elliptic-system}
 \left\{\begin{array}{ll}
         -\De u_j - u_j = \mu_ju_j^3+\be\sum_{k\ne j}u_k^2u_j,\\
         u_j>0\ \ \hbox{in}\ \Om,\ u_j=0 \ \ \hbox{on}\ \pa\Om,\ j=1, \dots, n.
        \end{array}
 \right.
\end{equation}
We also have some non-existence results for the general system \eqref{equ:es-1} complementing the main existence theorems.

In order to state our results we need some notation. We fix the parameters $\mu_1\le\dots\le\mu_n$. The scalar equation
\begin{equation}\label{equ:scalar}
 -\De\om-\om=-\om^3, \hspace{1cm} \om\in H_0^1(\Om).
\end{equation}
has a unique, non-degenerate solution $\om>0$, see \cite{Oruganti-Shi-Shivaji:2002} for details. A solution $(u_1,\dots,u_n)$ of \eqref{equ:es-2} is said to be synchronized if all components are positive multiples of $\om$, that is $u_j=\al_j\om$ with $\al_j>0$, all $j=1,\dots,n$. We consider the function
\begin{equation}\label{equ:function_g}
 g(\be) = 1+\be\sumj\frac{1}{\mu_j-\be},
\end{equation}
which is defined for $\be\in\R\setminus\{\mu_1,\dots,\mu_n\}$ and has the derivative
\begin{equation}\label{equ:g}
g'(\be)=\sumj\frac{\mu_j}{(\mu_j-\be)^2}.
\end{equation}
It has vertical asymptotes $\be=\mu_j$, $j=1,\dots,n$, and satisfies $\lim_{\be\to\pm\infty}g(\be)=1-n<0$. In the focusing case, $g$ satisfies $g'>0$ and $g(0)=1$. Consequently it has a unique zero $\bar\be$ in $(-\infty,0)$. In the defocusing case, $g$ satisfies $g'<0$ and $\lim_{\be\to\mu_n^+}g(\be)=\infty$, hence it has a unique zero $\bar\be$ in the interval $(\mu_n,\infty)$.
\begin{figure}[!ht]
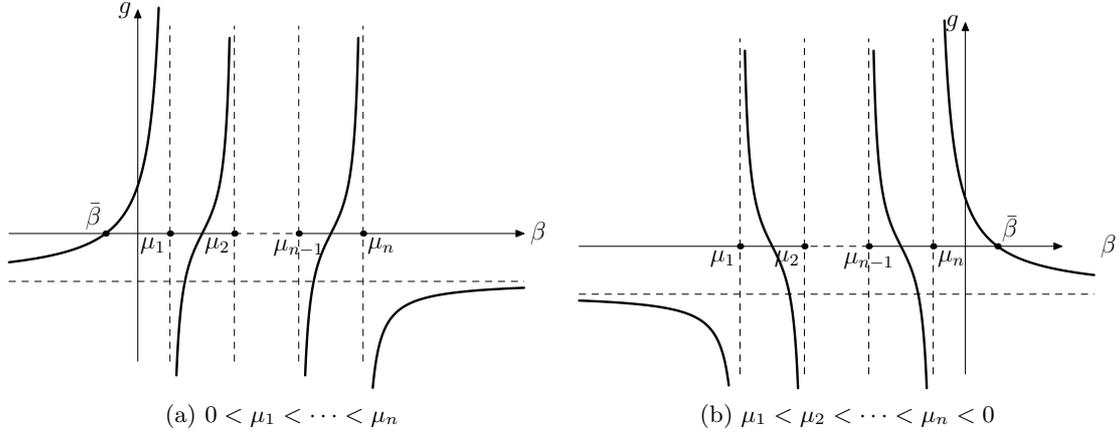

\centering
\subfloat[$0<\mu_1<\cdots<\mu_n$]{
\includegraphics[scale=0.8]{g_beta_focus.mps}
\label{fig:subfig-a}
}
\subfloat[$\mu_1<\mu_2<\cdots<\mu_n<0$]{
\includegraphics[scale=0.8]{g_beta_defocus.mps}
\label{fig:subfig-b}
}
\caption[Graphs of $g$]{Graphs of $g$ in the focusing case and defocusing case.}
\label{fig:bifurdiagram-ab}
\end{figure}

Now we can state our existence results.

\begin{theorem}\label{thm:synchr-branch}
 System \eqref{equ:es-2} has a synchronized solution branch
 \begin{equation}\label{equ:trivialbranch}
 \tw=\{(\be,u_1,\dots,u_n): u_j=\al_j(\be)\om,\ \be\in I\},
 \end{equation}
 which exists on the interval
 \begin{equation*}
  I =
  \left\{\begin{array}{ll}
          (-\infty, \bar{\be}) & \hbox{in the focusing case};\\
          (-\infty, \mu_1)\cup(\mu_n, \bar{\be}) & \hbox{in the defocusing case};\\
          (-\infty, \mu_1), & \hbox{in all the mixed cases}.
         \end{array}
  \right.
 \end{equation*}
 For $\be\in I$ the synchronized solution $\bfu(\be)=(u_1,\dots,u_n)\in[H_0^1(\Om)]^n$ is uniquely determined.
\end{theorem}

\begin{remark} a) If $n=2$, the parameter interval $I$ given above is the same as for the indefinite 2-equation system, see \cite{Tian-Wang:2013-jan, Tian-Wang:2013-feb} for details.

b) There exist more synchronized solutions if $\mu_j \equiv \mu=\be$; see \cite[Proposition~2.1]{Bartsch:2013} in the focusing case. There are also more synchronized solutions when one allows some components to be negative multiples of $\om$.
\end{remark}

Next we state our result about bifurcation points on $\tw$. For this the function
\[
f(\be) = -1-\frac2{g(\be)}
\]
and the scalar eigenvalue problem
\begin{equation}\label{equ:scalareigenpro}
 -\De\psi-\psi=\la\om^2\psi \ \ \hbox{in}\ \Om, \ \ \psi=0\ \ \hbox{on}\ \partial\Om,
\end{equation}
play an important role. Recall that the eigenvalue problem \eqref{equ:scalareigenpro} has an infinite sequence of eigenvalues: $-1=\la_1<\la_2<\cdots<\la_{k_0} <0<\la_{k_0+1}<\cdots$ and $\la_k\to\infty$ as $k\to\infty$.

\begin{theorem}\label{thm:bif}
 If $\be_k$ is a solution of the equation $f(\be)=\la_k$ then $(\be_k,\bfu(\be_k))\in\tw$ is a bifurcation point of \eqref{equ:es-2}. In the focusing case of system \eqref{equ:es-2}, the equation $f(\be)=\la_k$ has a unique solution for all but finitely many $k\in\N$, hence there are infinitely many bifurcation points on $\tw$. In the defocusing or mixed cases of \eqref{equ:es-2}, the equation $f(\be)=\la_k$ has a solution for at most finitely many $k\in\N$. The solution is unique in the defocusing case, whereas in the mixed cases it may have finitely many solutions; hence there are at most finitely many bifurcation points in both cases.
\end{theorem}

\begin{remark}
 The values of $\mu_j$ and the number of components determine the parameter interval of $\tw$ and affect the quantity of bifurcations along $\tw$. In particular, in the defocusing or mixed cases there may be no bifurcation points on $\tw$ depending on $n$ and $\Om$. In fact, fixing $\Om$ we shall see that there will be no solutions for $n$ large.
\end{remark}

For two-component systems, there is a global bifurcation branch emanating at every bifurcation point in the case $N=1$ or $\Om$ is radially symmetric, see \cite{Bartsch-Dancer-Wang:2010, Tian-Wang:2013-jan,Tian-Wang:2013-feb}. But for the multicomponent system, we do not have global bifurcation results in general, in particular for bifurcation solutions with all independent components. But, restricted to subspaces of $\R\times\h$ that possess the hidden symmetry as defined in \cite[Section 5]{Bartsch:2013}, global bifurcations may be found. Let $\cP=\{P_1, \cdots, P_m\}$ be a partition of $\{1, \dots, n\}$, $1\leq m\leq n$. If for any $1\leq j, k\leq n$ satisfying $j, k\in P_i$, $1\le i\le m$, the solution components $u_j$ and $u_k$ are synchronized, then the corresponding solution $\bfu$ is called a partially synchronized solution subject to partition $\cP$, or a $\cP$-synchronized solution for short.

About $\cP$-synchronized solutions we have the following theorem. Denote by $|\cP|$ the cardinality of $\cP$.

\begin{theorem}\label{thm:globalbifur}
 Let $\be_k$ be a solution of $f(\be)=\la_k$, so $(\be_k,\bfu(\be_k))\in\tw$ is a bifurcation point of \eqref{equ:es-2}.
 \begin{enumerate}[(i)]
  \item For every partition $\cP$ of $\{1, \dots, n\}$ with $|\cP| \ge 2$, $(\be_k,\bfu(\be_k))\in\tw$ is a bifurcation point of $\cP$-synchronized solutions of \eqref{equ:es-2}.
  \item Let $n_k$ denote the multiplicity of $\la_k$. If $(|\cP|-1)n_k$ is odd, then $(\be_k,\bfu(\be_k))\in\tw$ is a global bifurcation point of $\cP$-synchronized solutions.
  \item Suppose $n_k$ is odd. Let $A$ be a nonempty proper subset of $\{1,\dots,n\}$ and set $\cP_A=\{A,A^c\}$. Then there exists a global branch $\ms_k^A$ of $\cP_A$-synchronized solutions of \eqref{equ:es-2} bifurcating from $\tw$ at $(\be_k,\bfu(\be_k))$. Moreover, if $B$ is another nonempty subset of $\{1,\dots,n\}$, then the branches $\ms_k^A$ and $\ms_k^B$ are disjoint unless $A=B$ or $A=B^c$. In particular, there exist at least $2^{n_k-1}-1$ such global branches which are different.
  \item Let $A$ be a nonempty proper subset of $\{1,\dots,n\}$. If $N=1$ or $\Om$ is radial, then $\ms_k^A\cap\ms_l^A=\emptyset$ for $k\ne l$.
 \end{enumerate}
\end{theorem}

We also have some nonexistence results for solutions of the general system \eqref{equ:es-1}.

\begin{theorem}\label{thm:nonexistence}
 System \eqref{equ:es-1} does not have positive solutions in the following cases:
 \begin{enumerate}[(i)]
  \item if $a_j\le-\La_1$, $\mu_j>0$ for some $j=1,\dots,n$ and $\be\ge0$;
  \item if $a_j\le a_i$, $\mu_i\le\be\le\mu_j$ for some $i<j$ and at least one inequality holds strictly;
  \item in the focusing case, if $a_j\le-\La_1$ for all $j=1,\dots,n$, $\be\ge\bar{\be}$ and at least one inequality holds strictly;
  \item in the mixed cases, if $a_n\le a_1\le-\La_1$, $\be\ge\mu_1$ and at least one inequality holds strictly.
 \end{enumerate}
\end{theorem}

To close this section, we illustrate the local bifurcation results and the nonexistence results for the symmetric system with a few figures. In Figure~\ref{pic:bifurcationDiagrams}, the solid dots on $\tw$ are local bifurcation points and the shaded regions correspond to nonexistence intervals of $\beta$ for positive solutions of \eqref{equ:es-2}. The horizontal line $\cT_i$ represents the semi-trivial solution branch with only the $i$-th component being nontrivial. There are also semi-trivial solution branches with more nontrivial components. For example, all global bifurcation branches found in \cite{Tian-Wang:2013-jan, Tian-Wang:2013-feb} are semi-trivial solution branches with $2$ nontrivial components of \eqref{equ:es-2}. We omit them in Figure~\ref{pic:bifurcationDiagrams} to keep the diagrams clean.

\begin{figure}[!ht]
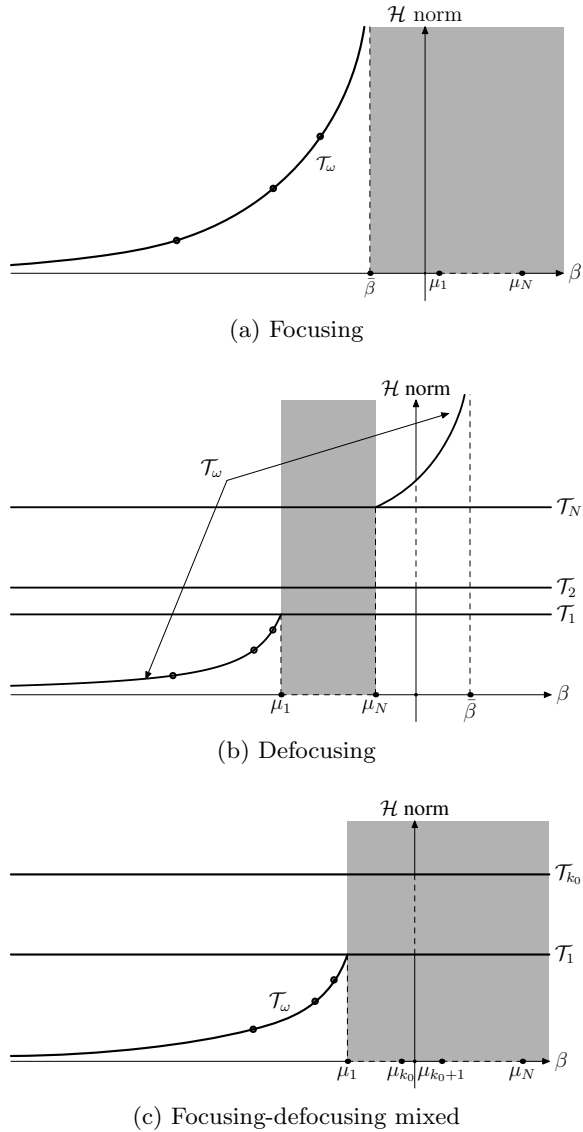

 \centering
 \subfloat[Focusing]{\includegraphics[width=0.45\textwidth]{local_BF_focus_mul.mps}}\\
 \subfloat[Defocusing]{\includegraphics[width=0.45\textwidth]{local_BF_defocus_mul.mps}}\\
 \subfloat[Focusing-defocusing mixed]{\includegraphics[width=0.45\textwidth]{local_BF_mix_mul.mps}}
 \caption{Nonexistence of positive solutions and local bifurcations of \eqref{equ:es-2}.}
 \label{pic:bifurcationDiagrams}
\end{figure}

\section{The synchronized solution branch}
In this section, we prove Theorem~\ref{thm:synchr-branch}. We make the ansatz
\begin{equation}\label{equ:u_j}
 u_j=\al_j\om,
\end{equation}
where the $\al_j$'s are positive constants. Substituting this into \eqref{equ:es-2}, we obtain the following system of equations for the coefficients $\al_j$:
\begin{equation*}
 \al_j(-\De\om-\om)
  = \al_j\left(\mu_j\al_j^2+\be\sum_{k\ne j}\al_k^2\right)\om^3.
\end{equation*}
Comparing with the scalar equation \eqref{equ:scalar}, we deduce
$\mu_j\al_j^2+\be\sum_{k\ne j}\al_k^2=-1$, which implies
\begin{equation*}
 (\mu_j-\be)\al_j^2=-1-\be\sum_{k=1}^n\al_k^2.
\end{equation*}
Note that the right-hand side of the above equation does not change in $j$, therefore
$$
(\mu_j-\be)\al_j^2=(\mu_k-\be)\al_k^2,\ \ \ \hbox{for } j,k=1,\dots, n.
$$
Substituting this and \eqref{equ:u_j} in the right-hand side of \eqref{equ:es-2} and combining like terms, we have
\begin{align*}
 -1 &= \mu_j\al_j^2+\be\sum_{k\ne j}\al_k^2
     = (\mu_j-\be)\al_j^2+\be\sum_{k=1}^n\al_k^2
     = (\mu_j-\be)\al_j^2+\be\sumk\frac{\mu_j-\be}{\mu_k-\be}\al_j^2\\
    &= (\mu_j-\be)\left(1+\be\sumk\frac{1}{\mu_k-\be}\right)\al_j^2
     = (\mu_j-\be)g(\be)\al_j^2.
\end{align*}
Consequently the system \eqref{equ:es-2} has a synchronized solution branch in the product space $\R\times\h$ provided
\begin{equation}\label{equ:cond-beta}
(\be-\mu_j)g(\be)>0\quad\text{ for all $j=1,\dots,n$.}
\end{equation}
Moreover, this branch is uniquely determined by setting $\al_j=\big((\be-\mu_j)g(\be)\big)^{-1/2}$ in \eqref{equ:u_j}.

Now we discuss  \eqref{equ:cond-beta} case by case. In the focusing case condition \eqref{equ:cond-beta} is satisfied precisely for $\be\in(-\infty,\bar{\be})$. In the defocusing case \eqref{equ:cond-beta} holds if, and only if, $\be\in(-\infty,\mu_1)\cup(\mu_n,\bar{\be})$. In the mixed cases $\mu_1\le\dots\le\mu_k<0<\mu_{k+1}\le\dots\le\mu_n$ with $1\le k\le n-1$, we consider the sign of $(\be-\mu_j)g(\be)$ by studying the auxiliary functions:
\begin{equation*}
 c_j(\be):=\frac{\be-\mu_j}{\prod_{k=1}^n(\mu_k-\be)}\quad\text{and}\quad
 G(\be):=\prod_{k=1}^n(\mu_k-\be)+\be\sum_{l=1}^n\prod_{k\ne l}(\mu_k-\be).
\end{equation*}
With these notations there holds $(\be-\mu_j)g(\be)=c_j(\be)G(\be)$.
For $\be\in(-\infty, \mu_1)$ one has $c_j(\be)<0$, $G(\mu_1)=\mu_1\prod_{k=2}^n(\mu_k-\mu_1)<0$, and
\begin{align*}
 G'(\be) &= -\sum_{j=1}^n\prod_{k\ne j}(\mu_k-\be)+\sum_{j=1}^n\prod_{k\ne j}(\mu_k-\be)-\be\sum_{j=1}^n\left(\sum_{i\ne j}\prod_{k\ne i, j}(\mu_k-\be)\right)\\
 &= -\be\sum_{j=1}^n\left(\sum_{i\ne j}\prod_{k\ne i, j}(\mu_k-\be)\right)>0.
\end{align*}
Thus $c_j(\be)G(\be)>0$ for all $j=1,\dots,n$.
For $\be\in(\mu_n, \infty)$ we distinguish between the cases $n$ being odd or even.
If $n$ is odd then $c_j(\be)<0$ for all $j=1,\dots,n$. Moreover,
$G(\mu_n) = \mu_n\prod_{k=1}^{n-1}(\mu_k-\mu_n) > 0$ and
\begin{equation*}
 G'(\be) = -\be\sum_{l=1}^n\left(\sum_{i\ne l}\prod_{k\ne i, l}(\mu_k-\be)\right) > 0.
\end{equation*}
This implies $c_j(\be)G(\be)<0$ for all $j=1,\dots,n$. In the case $n$ even we have $c_j(\be)>0$ for all $j=1,\dots,n$. Since $G(\mu_n) = \mu_n\prod_{k=1}^{n-1}(\mu_k-\mu_n) < 0$ and
\begin{equation*}
 G'(\be) = -\be\sum_{l=1}^n\left(\sum_{i\ne l}\prod_{k\ne i, l}(\mu_k-\be)\right) < 0
\end{equation*}
there holds $c_j(\be)G(\be) < 0$ for all $j=1,\dots,n$. Finally, for $\be\in(\mu_1,\mu_n)\setminus\{\mu_2,\dots,\mu_{n-1}\}$ we have that $\be-\mu_j$ is always positive for some $j$ and negative for the others. Therefore, for any fixed $\be$, there exist at least one $1\le j\le n$ such that $(\be-\mu_j)g(\be)<0$. In conclusion, in all mixed cases \eqref{equ:cond-beta} holds only for $\be\in(-\infty,\mu_1)$. This finishes the proof of Theorem~\ref{thm:synchr-branch}.

\section{The linearized system and possible bifurcation points}\label{sec:poss-bif}
In this section, we find all possible bifurcation parameters with respect to $\tw$, that is the values of $\be$ such that system \eqref{equ:es-2} has nontrivial kernel space. We consider the relaxed system
\begin{equation}\label{equ:EllipticsystemRelaxed}
 \left\{\begin{array}{ll}
         -\De u_j - u_j = \mu_ju_j^3+\be\sum_{k\ne j}u_k^2u_j,\\
         u_j=0 \ \ \hbox{on}\ \partial\Om,
        \end{array}
 \right.
\end{equation}
where we dropped the sign condition on the $u_j$'s. Local bifurcations of solution to \eqref{equ:EllipticsystemRelaxed} will be studied first, then using the Maximum Principle we will show that the bifurcating solutions are indeed positive, therefore they are also bifurcating solutions to \eqref{equ:es-2}.

We need to linearize system \eqref{equ:EllipticsystemRelaxed} at a solution $\bfu=(\al_1\om,\dots,\al_n\om)$ with $\al_j=[(\be-\mu_j)g(\be)]^{-1/2}$, in the direction
$\mathbf{\phi}=(\phi_1,\phi_2,\dots,\phi_n)\in\h$. Setting $\ga_j=\ga_j(\be)=(\mu_j-\be)^{-1/2}$ we compute:
{\allowdisplaybreaks
\begin{align*}
 -\De\phi_j-\phi_j
  &= 3\mu_j\al_j^2\om^2\phi_j + \be\sum_{k\ne j}\al_k^2\om^2\phi_j
      + 2\be\sum_{k\ne j}\al_j\al_k\om^2\phi_k\\
  &= \left(\frac{3\mu_j}{\be-\mu_j}
      + \be\sum_{k\ne j}\frac{1}{\be-\mu_k}\right)\frac{\om^2}{g(\be)}\phi_j
      + 2\be\sum_{k\ne j}\frac{\ga_j\ga_k}{g(\be)}\phi_k\om^2\\
  &= \left(\frac{2\mu_j}{\be-\mu_j} - 1 - \be\sumk\frac{1}{\mu_k-\be}\right)\frac{\om^2}{g(\be)}\phi_j
      + \frac{2\be\om^2}{g(\be)} \sum_{k\ne j}\ga_j\ga_k\phi_k\\
  &= \left(2\mu_j\ga_j^2-g(\be)\right)\frac{\om^2}{g(\be)}\phi_j
      + \frac{2\be\om^2}{g(\be)} \sum_{k\ne j}\ga_j\ga_k\phi_k\\
  &= \frac{2\om^2}{g(\be)}\left(\mu_j\ga_j^2\phi_j + \be\sum_{k\ne j}\ga_j\ga_k\phi_k\right)
      - \om^2\phi_j.
\end{align*}
}
Denote $C(\be)=\frac{2}{g(\be)}D(\be)-E_n$, where $E_n$ is the $n\times n$ identity matrix and
\begin{equation*}
 D(\be)=\left(\begin{array}{cccc}
        \mu_1\ga_1^2 & \be\ga_1\ga_2 & \cdots & \be\ga_1\ga_n\\
        \be\ga_2\ga_1 & \mu_2\ga_2^2 & \cdots & \be\ga_2\ga_n\\
        \vdots & \vdots & \ddots & \vdots\\
        \be\ga_n\ga_1 & \be\ga_n\ga_2 & \cdots & \mu_n\ga_n^2
       \end{array}
 \right).
\end{equation*} Then the linearized system becomes
\begin{equation}\label{equ:linearizedSystem}
 -\De \phi-\phi=\om^2C(\be)\phi.
\end{equation}
System \eqref{equ:linearizedSystem} must have a nontrivial solution $\phi$ in order that $\be$ is a bifurcation parameter. It is more convenient to rewrite system \eqref{equ:linearizedSystem} in terms of the eigenvectors of $C(\be)$ and then determine the possible bifurcation parameters by comparing with the scalar eigenvalue problem \eqref{equ:scalareigenpro}.

\begin{lemma}
 $C(\be)$ has the eigenvalue $-3$ and corresponding eigenvector
 $b_1(\be) = (\ga_1(\be),\dots,\ga_n(\be))^\top$.
\end{lemma}

\begin{proof} A direct calculation shows
{\allowdisplaybreaks
\begin{align*}
 D(\be)b_1(\be)&=\left(\begin{array}{cccc}
                           \mu_1\ga_1^2 & \be\ga_1\ga_2 & \cdots & \be\ga_1\ga_n\\
        \be\ga_2\ga_1 & \mu_2\ga_2^2 & \cdots & \be\ga_2\ga_n\\
        \vdots & \vdots & \ddots & \vdots\\
        \be\ga_n\ga_1 & \be\ga_n\ga_2 & \cdots & \mu_n\ga_n^2
                          \end{array}
 \right)\left(\begin{array}{c}
               \ga_1\\ \ga_2\\ \vdots\\ \ga_n
              \end{array}
\right)\\
&=\left(\begin{array}{c}
        \ga_1(\mu_1\ga_1^2+\be\sum_{k\ne1}\ga_k^2)\\
        \ga_2(\mu_2\ga_2^2+\be\sum_{k\ne2}\ga_k^2)\\
        \vdots\\
        \ga_n(\mu_n\ga_n^2+\be\sum_{k\ne n}\ga_k^2)
       \end{array}
 \right)=\left(\begin{array}{c}
        \ga_1[(\mu_1-\be)\ga_1^2+\be\sumk\ga_k^2)]\\
        \ga_2[(\mu_2-\be)\ga_2^2+\be\sumk\ga_k^2)]\\
        \vdots\\
        \ga_1[(\mu_n-\be)\ga_n^2+\be\sumk\ga_k^2)]\\
       \end{array}
 \right)\\
&=\left(\begin{array}{c}
         \ga_1[-1-\be\sumk(\mu_k-\be)^{-1}]\\
         \ga_2[-1-\be\sumk(\mu_k-\be)^{-1}]\\
         \vdots\\
         \ga_n[-1-\be\sumk(\mu_k-\be)^{-1}]
        \end{array}
 \right)=-g(\be)\left(\begin{array}{c}
                         \ga_1\\
                         \ga_2\\
                         \vdots\\
                         \ga_n
                        \end{array}
                  \right)=-g(\be)b_1(\be).
\end{align*}}
Therefore $C(\be)b_1(\be)=-2b_1(\be)-b_1(\be)=-3b_1(\be)$.
\end{proof}

\begin{lemma}
  The number $f(\be)=-\frac{2}{g(\be)}-1$ is an eigenvalue of $C(\be)$ with multiplicity $n-1$.
\end{lemma}

\begin{proof} We define $n-1$ linearly independent vectors $b_j(\be)=(b_{j1},\dots,b_{jn})^\top$, $j=2,\dots,n$, as follows:
\begin{center}
 $b_{j1}=\ga_j$, $b_{jj}=-\ga_1$ and $b_{jk}=0$ for $k\ne1$ and $k\ne j$.
\end{center}
Clearly $b_j(\be)$ is orthogonal to $b_1(\be)$ for $j=2,\dots,n$. Applying $D(\be)$ to $b_j(\be)$ for $j\ge2$ yields
{\allowdisplaybreaks
\begin{align*}
 D(\be)b_j(\be)&=\left(\begin{array}{cccc}
                           \mu_1\ga_1^2 & \be\ga_1\ga_2 & \cdots & \be\ga_1\ga_n\\
        \be\ga_2\ga_1 & \mu_2\ga_2^2 & \cdots & \be\ga_2\ga_n\\
        \vdots & \vdots & \ddots & \vdots\\
        \be\ga_n\ga_1 & \be\ga_n\ga_2 & \cdots & \mu_n\ga_n^2
                          \end{array}
 \right)\left(\begin{array}{c}
               \ga_j\\ \vdots\\ -\ga_1\\ \vdots\\ 0
              \end{array}
\right)=\left(\begin{array}{c}
               (\mu_1-\be)\ga_1^2\ga_j\\
               \vdots\\
               (\be-\mu_j)\ga_j^2\ga_1\\
               \vdots\\
               0
              \end{array}
        \right)=-b_j(\be).
\end{align*}
}
Consequently,
$$
C(\be)b_j(\be) = \left(\frac{2}{g(\be)}D(\be)-E_n\right)b_j(\be)
 = \left(-\frac{2}{g(\be)}-1\right)b_j(\be) = f(\be)b_j(\be),
$$
hence $f(\be)$ is an eigenvalue of $C(\be)$ with $n-1$ eigenvectors $b_j(\be)$, $2\le j\le n$.
\end{proof}

Since $C(\be)$ is a real symmetric matrix, we can use its eigenvectors to construct an orthogonal matrix $T(\be)$ which diagonalizes $C(\be)$, that is $T^{-1}(\be)C(\be)T(\be)=\hbox{diag}(-3, f(\be), \dots, f(\be))$. Moreover, since $b_j(\be)$ depends smoothly on $\be$ we may assume that $T(\be)$ also depends smoothly on $\be$. The linearized system of \eqref{equ:EllipticsystemRelaxed} now is equivalent to
\begin{equation}\label{equ:equivlinearized}
 \left\{\begin{array}{l}
        -\De\psi_1-\psi_1=-3\om^2\psi_1,\\
        -\De\psi_j-\psi_j=f(\be)\om^2\psi_j,\ j=2,\dots,n.
       \end{array}
 \right.
\end{equation}
The principal eigenvalue of \eqref{equ:scalareigenpro} is $-1$, thus the first equation of \eqref{equ:equivlinearized} only has the zero solution. As a result, a nontrivial solution component of \eqref{equ:equivlinearized} must come from the remaining $n-1$ equations. Thus we need to find all solutions of the equations $f(\be)=\la_k$, $k\ge1$. Since the number and the location of the bifurcation parameters depend on the $\mu_j$'s, we will find the local bifurcations case by case.

\begin{lemma}\label{lemma:fParametersFocusing}
In the focusing case $0<\mu_1\le\cdots\le\mu_n$ there are infinitely many possible bifurcation parameters. More precisely, these parameters are determined by the equations $f(\be)=\la_k$ which has a (unique) solution for all but a finite number of $k\in\mathbb{N}$.
\end{lemma}

\begin{proof} In this case, $\tw$ exists on the interval $(-\infty, \bar{\be})$, where $\bar{\be}$ is the unique value of $g(\be)=0$ in $(-\infty, \mu_1)$. As proved above the equation $f(\be)=-1-\frac{2}{g(\be)} = \la_k$ on $(-\infty, \bar{\be})$ determines the bifurcation parameters.

Note that $f$ is a rational function and is smooth on $(-\infty, \bar{\be})$ with vertical asymptote $\be=\bar{\be}$. Recall that
\begin{equation*}
 \lim_{\be\to-\infty}g(\be) = 1-n < 0,\hspace{0.5cm}
 g(\bar\be) = 0,\hspace{0.5cm}
 g'(\be)=\sumk\frac{\mu_k}{(\mu_k-\be)^2} > 0,
\end{equation*}
therefore
\begin{equation*}
 \lim_{\be\to-\infty}f(\be) = -1-\frac{2}{1-n},\hspace{0.5cm}
 \lim_{\be\to\bar{\be}^-}f(\be) = \infty,\hspace{0.5cm}
 f'(\be)=\frac{2g'(\be)}{[g(\be)]^2} > 0.
\end{equation*}
According to the behavior of $f$, $f(\be)=\la_k$ has a unique solution for all $\la_k$ satisfying $$\la_k>-1-\frac{2}{1-n}=\frac{3-n}{n-1}.$$ Since $\la_k\to\infty$, this inequality is satisfied for all but finitely many values of $k\in\N$.
\end{proof}

\begin{lemma}\label{lemma:fParametersDefocusing}
 In the defocusing case $\mu_1\le\dots\le\mu_n<0$ there are at most finitely many possible bifurcation parameters.
\end{lemma}

\begin{proof} In this case, the synchronized solution branch $\tw$ exists for $\be\in(-\infty,\mu_1)\cup(\mu_n,\bar{\be})$, where $\bar{\be}$ is the unique number in $(\mu_n,\infty)$ such that $g(\bar{\be})=0$. In the interval $(-\infty, \mu_1)$, we have
\begin{equation*}
 \lim_{\be\to-\infty}g(\be)=1-n,\hspace{0.5cm}
 \lim_{\be\to\mu_1^-}g(\be)=-\infty,\hspace{0.5cm}
 g'(\be)=\sumk\frac{\mu_k}{(\mu_k-\be)^2}<0.
\end{equation*}
According to the relation $f(\be)=-1-\frac{2}{g(\be)}$, there holds
\begin{equation*}
 \lim_{\be\to-\infty}f(\be)=-1+\frac{2}{n-1},\hspace{0.5cm}
 \lim_{\be\to\mu_1^-}f(\be)=-1,\hspace{0.5cm}
 f'(\be)=\frac{2g'(\be)}{[g(\be)]^2}<0.
\end{equation*}
Therefore the equation $f(\be)=\la_k$ can only be solved if
\begin{equation}\label{equ:lambda_k}
-1<\la_k<-1+\frac{2}{n-1}.
\end{equation}
The monotonicity of $f$ also implies that $f(\be)=\la_k$ has at most one solution for each $k$.

In the interval $(\mu_n,\bar{\be})$ we have
\begin{equation*}
 \lim_{\be\to\mu_n^+}g(\be)=\infty,\hspace{0.5cm}
 \lim_{\be\to\bar{\be}^-}g(\be)=0,\hspace{0.5cm}
 g'(\be)=\sumk\frac{\mu_k}{(\mu_k-\be)^2}<0.
\end{equation*}
Accordingly, we obtain
\begin{equation*}
 \lim_{\be\to\mu_n^+}f(\be) = -1,\hspace{0.5cm}
 \lim_{\be\to\bar{\be}^-}f(\be) = -\infty,\hspace{0.5cm}
 f'(\be) = \frac{2g'(\be)}{[g(\be)]^2} < 0.
\end{equation*}
Since all eigenvalues of \eqref{equ:scalareigenpro} are greater than or equal to $-1$, there is no bifurcation parameter in the interval $(\mu_n,\bar{\be})$.
\end{proof}

\begin{remark}
Equation \eqref{equ:lambda_k} may have no solution at all. If $n\to\infty$ the range for the eigenvalues $\la_k$ of \eqref{equ:scalareigenpro} to satisfy \eqref{equ:lambda_k} becomes smaller and, for $n$ large no eigenvalue satisfies \eqref{equ:lambda_k}.
\end{remark}

\begin{lemma}\label{lemma:fParametersMixed}
 In all mixed cases there are at most finitely many possible bifurcation parameters.
\end{lemma}

\begin{proof} In all mixed cases $\tw$ exists for $\be\in(-\infty, \mu_1)$. Here we have
\begin{equation*}
 \lim_{\be\to-\infty}f(\be)=-1+\frac{2}{n-1},\hspace{0.5cm}
 \lim_{\be\to\mu_1^-}f(\be)=-1.
\end{equation*}
Similar to the defocusing case, the equation $f(\be)=\la_k$ has solutions only if $-1<\la_k<-1+\frac{2}{n-1}$, hence for at most finitely many $k$.

Observe that $f$ is not a monotone function in the mixed cases. If $f$ has a maximum value
$f_{max} > -1+\frac{2}{n-1}$, then for a fixed $\la_k$ satisfying $-1+\frac{2}{n-1} < \la_k < f_{max}$, the continuity of $f$ implies that the graph of $f(\be)$ will cross the horizontal line $\la=\la_k$ more than once, that is $f(\be)=\la_k$ has more than one solutions.
\end{proof}

\begin{remark}\label{remark:fprimeThreecases}
 In the focusing case and in the defocusing case, $f'(\be)$ has a fixed sign and is never zero in $I$. In the mixed cases, if $\sumk \mu_k < 0$, then it is easy to see that $f'(\be) \ne 0$ for $\be\in(-\infty, \mu_1)$.  But if the sum $\sumk\mu_k > 0$, then there may exist $\be\in(-\infty, \mu_1)$ such that $f'(\be)=0$. As we shall see in the next section, these facts are important in verifying local bifurcations.
\end{remark}

\section{The verification of local bifurcations}
Let $\be_k$ denote a solution of $f(\be)=\la_k$. Recall that $\be_k$ is uniquely defined in the focusing or defocusing case but not in the mixed cases. The fact that \eqref{equ:EllipticsystemRelaxed} has a nonempty kernel at $\be=\be_k$ is not sufficient to claim a bifurcation point. In this section, we use \cite[Theorem 8.9]{Mawhin-Willem:1989} to verify that these $\be_k$'s are indeed bifurcation parameters.

We denote the eigenspace of the linear eigenvalue problem \eqref{equ:scalareigenpro} by
\begin{equation*}
 V_k=\{\phi\in H_0^1(\Om)\ |\ -\Delta\phi-\phi=\la_k\omega^2\phi\},\ \ k=1,2,\dots
\end{equation*}
and set $n_k=\dim V_k$. The following lemma can be established with similar arguments as \cite[Lemma~4.1]{Bartsch:2013}. We include the proof here for the convenience of the reader.

\begin{lemma}\label{lem:morseindexDifference}
 $\be_k$ is a bifurcation parameter if $f'(\be_k)\neq0$.
\end{lemma}

\begin{proof} Let $J_\be: \h\to \R$ be the energy functional associated with \eqref{equ:EllipticsystemRelaxed}, that is
\begin{equation}
 J_\be(u_1,\ldots,u_n)=\frac{1}{2}\sumk\int_\Om(|\nabla u_k|^2-u_k^2)-\frac{1}{4}\sumk\int_\Om\mu_ku_k^4-\frac{\be}{2}\sum_{i<k}\int_\Om u_i^2u_k^2.
\end{equation}
By Sobolev embedding, $J_\be$ is well-defined and of class $C^2$. We can calculate the Morse index $m(\be)$ of $J_\be$ at $(\be,\bfu(\be))\in\tw$, in particular near the possible bifurcation points $(\be_k,\bfu(\be_k))$ found in Section~\ref{sec:poss-bif}. According to \cite[Theorem~8.9]{Mawhin-Willem:1989}, if the Morse index changes as $\be$ passes $\be_k$, then $(\be_k,\bfu(\be_k))$ is a bifurcation point. More precisely, we claim:
\begin{equation}\label{equ:morseindexDifference}
  |m(\be_k-\eps)-m(\be_k+\eps)|=(n-1)n_k,
\end{equation}
provided $f'(\be_k)\neq0$ and $\eps>0$ small. If \eqref{equ:morseindexDifference} is established, then the lemma is proved.

Denote the Hessian of $J_\be$ at $\bfu(\be)$ by
\begin{equation*}
 Q_\be(\bfv) := \langle J_\be''(\bfu(\be))\bfv,\bfv\rangle
   = \sumk\int_\Om(|\nabla v_k|^2-v_k^2)-\int_\Om \om^2\langle C(\be)\bfv,\bfv\rangle.
\end{equation*}
The Morse index $m(\be)$ is the dimension of the negative eigenspace of $Q_\be$. We decompose the space $\h=V_{\be_k}^-\oplus V_{\be_k}^0\oplus V_{\be_k}^+$, where $V_{\be_k}^-$, $V_{\be_k}^+$ and $V_{\be_k}^0$ are the negative eigenspace, positive eigenspace and kernel of $Q_{\be_k}$ respectively. In particular,
\begin{equation*}
 V_{\be_k}^0 = \left\{\bfv\in \h: v_j\in V_k\ \hbox{for}\ j=1,\dots,n,\ \sumj\ga_j(\be)v_j=0\right\},
\end{equation*}
hence $\dim V_{\be_k}^0=(n-1)n_k$. Since $C(\be)$ is a smooth function of $\be$, we have the expansion
\begin{equation*}
 Q_\be = Q_{\be_k}+(\be-\be_k)Q_{\be_k}'+o(|\be-\be_k|)\ \ \ \hbox{as}\ \be\to\be_k.
\end{equation*}
This implies that $Q_\be > 0$ on $V_{\be_k}^+$ and $Q_\be < 0$ on $V_{\be_k}^-$, provided $\be$ is close to $\be_k$. Thus the claim is true if
\begin{equation*}
 Q_{\be_k}'(\bfv) = -\int_\Om \om^2\langle C'(\be_k)\bfv,\bfv\rangle,
\end{equation*}
is positive, or negative, definite on ${V_{\be_k}^0}$.

Since $C(\be)$ is a real symmetric matrix, there exists an orthogonal matrix $T(\be)$, depending smoothly on $\be$, such that
\begin{equation*}
 T^{-1}(\be)C(\be)T(\be)=\hbox{diag}(-3, f(\be), \cdots, f(\be))=:C_T(\be),
\end{equation*}
which is equivalent to $C(\be)T(\be)=T(\be)C_T(\be)$. Differentiating both sides with respect to $\be$ and then rearranging terms leads to
\begin{equation*}
 C'(\be)=T'(\be)C_T(\be)T^{-1}(\be)+T(\be)C_T'(\be)T^{-1}(\be) - C(\be)T'(\be)T^{-1}(\be).
\end{equation*}
For any $\bfv\in V_{\be_k}^0$,
\begin{align*}
 \abr{C'(\be)\bfv,\bfv}
  &= \abr{T'(\be)C_T(\be)T^{-1}(\be)\vv,\vv} + \abr{T(\be)C_T'(\be)T^{-1}(\be)\vv,\vv}
      - \abr{C(\be)T'(\be)T^{-1}(\be)\vv,\vv}\\
  &= \abr{T'(\be)T^{-1}(\be)C(\be)\vv,\vv} + f'(\be)\abr{T(\be)T^{-1}(\be)\vv,\vv}
      - \abr{T'(\be)T^{-1}(\be)\vv,C(\be)\vv}\\
  &= f(\be)\abr{T'(\be)T^{-1}(\be)\vv,\vv}+f'(\be)|\vv|^2 - f(\be)\abr{T'(\be)T^{-1}(\be)\vv,\vv}\\
  &= f'(\be)|\vv|^2.
\end{align*}
Therefore, $Q_{\be_k}'(\bfv)$ is positive definite on $V_{\be_k}^0$ if $f'(\be_k)>0$, or negative definite on ${V_{\be_k}^0}$ if $f'(\be_k)<0$. The claim \eqref{equ:morseindexDifference} follows and the lemma is proved.
\end{proof}

\begin{altproof}{Theorem~\ref{thm:bif}}
For the relaxed system \eqref{equ:EllipticsystemRelaxed} in the focusing, defocusing or mixed cases, the lemmas \ref{lemma:fParametersFocusing}, \ref{lemma:fParametersDefocusing} and \ref{lemma:fParametersMixed}, respectively, yield an infinite sequence of possible bifurcation points in the focusing case, and finitely many bifurcation points in the other cases. According to Lemma
\ref{lem:morseindexDifference}, local bifurcation for the relaxed system  \eqref{equ:EllipticsystemRelaxed} occurs at each $\be_k$, provided $f'(\be_k)\neq0$. By Remark~\ref{remark:fprimeThreecases}, this inequality is always satisfied in the focusing case and defocusing case. It may fail in some mixed cases.

At last, we need to show that the bifurcating solutions of \eqref{equ:EllipticsystemRelaxed} are positive. Notice that
\begin{equation}\label{equ:omegaproperty}
 \om>0\ \ \ \hbox{in}\ \Om\quad \hbox{and}\quad \frac{\pa\om}{\pa\nu}<0\ \ \ \hbox{on}\ \pa\Om,
\end{equation}
where $\nu$ is the unit outward normal vector on $\pa\Om$. According to Sobolev embeddings and elliptic regularity theory, the bifurcating solutions that are close enough to a solution on $\tw$ in the $H_0^1$-norm are also close to the same solution on $\tw$ in the $C^1$-norm.  Then \eqref{equ:omegaproperty} implies that the bifurcating solutions are positive in $\Om$.
\end{altproof}

\begin{remark}
 According to the bifurcation theory, $(\be_k, u_1(\be_k),\dots,u_n(\be_k))$ is a global bifurcation point if $(n-1)n_k$ is odd. If $n=2$ and $n_k=1$, which holds for $N=1$ or a radially symmetric domain $\Om$, the Crandall-Rabinowitz theorem applies and yields locally a smooth curve of bifurcating solutions. In the other cases, the change of Morse index is greater than one, so we cannot obtain further information about the global bifurcation branches by the arguments used in \cite{Bartsch-Dancer-Wang:2010, Tian-Wang:2013-jan, Tian-Wang:2013-feb}. Some general information on the bifurcating branches can be deduced from Dancer's analytic bifurcation theory \cite{dancer:1973}.
\end{remark}

\begin{remark}
 Using the results of \cite{Tian-Wang:2013-jan, Tian-Wang:2013-feb}, system \eqref{equ:es-2} always has semi-trivial solution branches with two nonzero components, provided $n\ge3$. But semi-trivial solution branch with exactly one nonzero component do not exist in the focusing case.
\end{remark}

%================================
\section{Partially synchronized solutions and global bifurcations}
%===============================
In this section, we study the global bifurcation phenomena of partially synchronized solutions of \eqref{equ:es-2}, and prove Theorem \ref{thm:globalbifur}.

In contrast with \cite{Bartsch-Dancer-Wang:2010, Tian-Wang:2013-jan, Tian-Wang:2013-feb}, we cannot claim the existence of global bifurcation at $\be_k$ even when $n_k$ is odd, since \eqref{equ:morseindexDifference} shows that the Crandall-Rabinowitz condition for global bifurcation now also depends $n$. In particular, if $n$ is odd, then $(n-1)n_k$ is even and no global bifurcation at $\be_k$ can be claimed. Using the hidden symmetry observed in \cite{Bartsch:2013}, we may find global bifurcations in some subspaces of $\R\times\h$.

\begin{altproof}{Theorem~\ref{thm:globalbifur}}
It is straightforward to check that the following results cited from \cite{Bartsch:2013} can be applied to the system \eqref{equ:es-2} with no substantial changes.
\begin{enumerate}[(i)]
  \item On one hand, the $|\cP|$-synchronized solutions of system \eqref{equ:es-2} satisfy a reduced system of \eqref{equ:es-2} with $|\cP|$ components, see \cite[Lemma 5.1]{Bartsch:2013}. On the other hand, using the solution of this reduced $|\cP|$-component system, we can construct a $|\cP|$-synchronized solution of system \eqref{equ:es-2}, see \cite[Lemma 5.2 and Proposition 5.3]{Bartsch:2013}. Next, bifurcations of $|\cP|$-synchronized solutions can be verified at bifurcation points of general solutions, see \cite[Lemma 5.4]{Bartsch:2013}. Thus (i) is proved.
  \item If $(|\cP|-1)n_k$ is odd, then by Crandall-Rabinowitz's bifurcation theorem, global bifurcations of $\cP$-synchronized solutions occur at each bifurcation point $(\be_k, \bfu(\be_k))$.
  \item If $n_k$ is odd and $A$ is a nonempty proper subset of $\{1,\dots,n\}$, then the existence of a global bifurcation branch $\ms_k^A$ with $\cP_A$-synchronized solutions can be easily seen for $\cP_A=\{A, A^c\}$. Let $B$ be another nonempty proper subset of $\{1,\dots,n\}$ satisfying $B\ne A$ and $B\ne A^c$. If $\ms_k^A\cap\ms_k^B\ne\emptyset$, then there exists $(\be,\bfu)\in\ms_k^A\cap\ms_k^B$. By the definition of partially synchronized solution and simple set operations, we obtain that all components of $\bfu$ are synchronized. This contradicts with the fact that $A$ and $B$ are both nonempty proper subsets of $\{1,\dots,n\}$.
  \item Let $\cP=\{A, A^c\}$. In the case $N=1$ or $\Om$ is radially symmetric, $n_k=1$ for every eigenvalue $\la_k$ of \eqref{equ:scalareigenpro}. Thus there is a global bifurcation branch of $\cP$-synchronized solutions at each bifurcation point $(\be_k,\bfu(\be_k))$. The conclusion follows from the the bifurcation results for indefinite two-component systems, see \cite{Tian-Wang:2013-jan, Tian-Wang:2013-feb}.
\end{enumerate}
\end{altproof}

\section{Nonexistence results}
In this section we prove Theorem~\ref{thm:nonexistence}. We argue by contradiction and assume that $\bfu$ is a solution of \eqref{equ:es-1}.

\begin{altproof}{Theorem~\ref{thm:nonexistence}(i)}
We multiply the $j$-th equation in \eqref{equ:es-2} with the principal eigenfunction $\phi_1$ of $-\De$ in $H^1_0(\Om)$ and obtain:
\[
0 \ge (\La_1+a_j)\int_\Om u_j\phi_1 = \int_\Om (-\De u_j+a_ju_j)\phi_1\\
  = \int_\Om \left(\mu_ju_j^3+\be\sum_{k\ne j}u_k^2\right)\phi_1 > 0
\]
\end{altproof}

\begin{altproof}{Theorem~\ref{thm:nonexistence}(ii)}
Here we multiply the $i$-th equation by $u_j$, the $j$-th equation by $u_i$, and obtain:
\[
\ba
 0 \geq (a_j-a_i)\int u_iu_jdx = (\mu_j-\be)\int u_iu_j^3+(\be-\mu_i)\int u_i^3u_j \ge 0.
\ea
\]
This gives a contradiction if one of the inequalities $a_j\le a_i$, $\mu_i\le\be\le\mu_j$ is strict.
\end{altproof}

\begin{altproof}{Theorem~\ref{thm:nonexistence}(iii)}
By Theorem~\ref{thm:nonexistence}(i) we only need to consider the case $\bar\be < \be \le 0$. The following argument works for $\bar\be < \be < \mu_1$. Multiplying both sides of the $k$-th equation by $\al_k\phi_1$, where $\al_k=(\sqrt{\mu_k-\be})^{-1}$, we obtain:
{\allowdisplaybreaks
\begin{align*}
0 &\geq \sumk\int(\La_1+a_k)\al_ku_k\phi_1 = \int\left(\sumk\mu_k\al_ku_k^3
    + \be\sumk\al_ku_k\sum_{l\ne k}^nu_l^2\right)\phi_1\\
  &= \int\left(\sumk(\mu_k-\be)\al_ku_k^3 + \be\sumk\al_ku_k\sumk u_k^2\right)\phi_1\\
  &= \left(\sumk\frac{1}{\mu_k-\be}\right)^{-1}
      \int\left(\sumk\frac{1}{\mu_k-\be}\sumk(\mu_k-\be)\al_ku_k^3 +
                \sumk\frac{\be}{\mu_k-\be}\sumk\al_ku_k\sumk u_k^2\right)\phi_1
\end{align*}}
We only need to make sure that the summation inside the parentheses is positive. Using the definition of $g(\beta)$ and of $\al_k$ we obtain:
{\allowdisplaybreaks
\begin{align*}
&\sumk\frac{1}{\mu_k-\be}\sumk(\mu_k-\be)\al_ku_k^3
   + \sumk\frac{\be}{\mu_k-\be}\sumk\al_ku_k\sumk u_k^2\\
&\hspace{.5cm}
 = g(\be)\sumk\al_ku_k^3
    + \sum_{i<j}\left(\frac{\mu_j-\be}{\mu_i-\be}\al_ju_j^3 + \frac{\mu_i-\be}{\mu_j-\be}\al_iu_i^3
      + \sumk\frac{\be}{\mu_k-\be}(\al_iu_iu_j^2+\al_ju_i^2u_j)\right)\\
&\hspace{.5cm}
 = g(\be)\left(\sumk\al_ku_k^3 + \sum_{i<j}(\al_iu_iu_j^2+\al_ju_i^2u_j)\right)\\
&\hspace{1cm}
   +\sum_{i<j}\left(\frac{\mu_j-\be}{\mu_i-\be}\al_ju_j^3 +
                    \frac{\mu_i-\be}{\mu_j-\be}\al_iu_i^3-(\al_iu_iu_j^2+\al_ju_i^2u_j)\right)\\
&\hspace{.5cm}
 = g(\be)\left(\sumk\al_ku_k^3 + \sum_{i<j}(\al_iu_iu_j^2+\al_ju_i^2u_j)\right)\\
&\hspace{1cm}
   + \sum_{i<j}\left[u_j^2\left(\frac{\mu_j-\be}{\mu_i-\be}\al_ju_j-\al_iu_i\right)
   + u_i^2\left(\frac{\mu_i-\be}{\mu_j-\be}\al_iu_i-\al_ju_j\right)\right]\\
&\hspace{.5cm}
 = g(\be)\left(\sumk\al_ku_k^3 + \sum_{i<j}(\al_iu_iu_j^2+\al_ju_i^2u_j)\right)\\
&\hspace{1cm}
   + \sum_{i<j}\left[\frac{\mu_j-\be}{\mu_i-\be}u_j^2
                      \left(\al_ju_j-\frac{\mu_i-\be}{\mu_j-\be}\al_iu_i\right)
                     + u_i^2\left(\frac{\mu_i-\be}{\mu_j-\be}\al_iu_i-\al_ju_j\right)\right]\\
&\hspace{.5cm}
 = g(\be)\left(\sumk\al_ku_k^3 + \sum_{i<j}(\al_iu_iu_j^2+\al_ju_i^2u_j)\right)
    + \sum_{i<j}\left(\frac{\mu_j-\be}{\mu_i-\be}u_j^2-u_i^2\right)
        \left(\al_ju_j-\frac{\mu_i-\be}{\mu_j-\be}\al_iu_i\right)\\
&\hspace{.5cm}
 = g(\be)\left(\sumk\al_ku_k^3 + \sum_{i<j}(\al_iu_iu_j^2+\al_ju_i^2u_j)\right)\\
&\hspace{1cm}
    + \sum_{i<j}\left(\sqrt{\frac{\mu_j-\be}{\mu_i-\be}}u_j+u_i\right)
        \left(\sqrt{\frac{\mu_j-\be}{\mu_i-\be}}u_j-u_i\right)
        \left(\frac{(\mu_j-\be)\al_j}{(\mu_i-\be)\al_i}u_j-u_i\right)\frac{\mu_i-\be}{\mu_j-\be}\al_i\\
&\hspace{.5cm}
 = g(\be)\left(\sumk\al_ku_k^3 + \sum_{i<j}(\al_iu_iu_j^2+\al_ju_i^2u_j)\right)\\
&\hspace{1cm}
    + \sum_{i<j}\left(\sqrt{\mu_j-\be}u_j + \sqrt{\mu_i-\be}u_i\right)
                \left(\frac{u_j}{\sqrt{\mu_i-\be}}-\frac{u_i}{\sqrt{\mu_j-\be}}\right)^2\\
&\hspace{.5cm}
 \geq g(\be)\left(\sumk\al_ku_k^3 + \sum_{i<j}(\al_iu_iu_j^2+\al_ju_i^2u_j)\right).
\end{align*}}
Substituting the above inequality in the previous integral inequality yields a contradiction if $a_j<\La_1$ for some $j$ or $g(\be)>0$.
\end{altproof}

\begin{altproof}{Theorem~\ref{thm:nonexistence}(iv)}
The case $\mu_1\le\be<\mu_n$ has been treated in (ii), the case $\be\ge\mu_n>0$ in (i).
\end{altproof}


\begin{thebibliography}{9}
 \vspace{-1mm}
\thispagestyle{myheadings}
{
\footnotesize

\bibitem{Ambrosetti-Colorado:2006}
  \newblock {A. Ambrosetti and E. Colorado},
  \newblock \emph{Bound and ground states of coupled nonlinear Schr\"{o}dinger equations}.
  \newblock C. R. Math. Acad. Sci. Paris \textbf{342} (2006), 453-458.

\bibitem{Ambrosetti-Colorado:2007}
  \newblock {A. Ambrosetti and E. Colorado},
  \newblock \emph{Standing waves of some coupled nonlinear Schr\"{o}dinger equations}.
  \newblock J. Lond. Math. Soc. \textbf{75} (2007), 67-82.

\bibitem{Bartsch:2013}
 \newblock {T. Bartsch},
 \newblock \emph{Bifurcation in a multicomponent system of nonlinear Schr\"odinger equations}.
 \newblock J. Fixed Point Theory Appl. \textbf{13} (2013), 37-50.

\bibitem{Bartsch-Dancer-Wang:2010}
 \newblock {T. Bartsch, E.N. Dancer and Z.-Q. Wang},
 \newblock\emph{A Liouville theorem, a-priori bounds, and bifurcating branches of
positive solutions for a nonlinear elliptic system}.
 \newblock Calc. Var. Part. Diff. Equ. {\bf 37} (2010), 345-361.

\bibitem{Bartsch-Wang:2006}
  \newblock {T. Bartsch and Z.-Q. Wang},
  \newblock \emph{Note on ground states of nonlinear Schr\"{o}dinger systems}.
  \newblock J. Part. Diff. Equ. \textbf{19} (2006), 200-207.

\bibitem{Bartsch-Wang-Wei:2007}
  \newblock {T. Bartsch, Z.-Q. Wang and J. Wei},
  \newblock \emph{Bound states for a coupled Schr\"{o}dinger system}.
  \newblock J. Fixed Point Theory Appl. \textbf{2} (2007), 353-367.

\bibitem{dancer:1973}
 \newblock {E.N. Dancer},
 \newblock \emph{Global structure of the solutions of non-linear real analytic eigenvalue problems}.
 \newblock Proc. London Math. Soc. (3) \textbf{27} (1973), 747-765.

\bibitem{Dancer-Wei-Weth:2010}
 \newblock {E.N. Dancer, J.C. Wei and T. Weth},
 \newblock \emph{A priori bounds versus multiple existence of positive solutions for a
nonlinear Schr\"{o}dinger system}.
 \newblock Ann. Inst. H. Poincar\'e Anal. Non Lin\'eaire \textbf{27} (2010), 953-969.


\bibitem{Esry-Greene-Burke-Bohn:1997}
 \newblock {B.D. Esry, C.H. Greene, J.P. Burke Jr and J.L. Bohn},
 \newblock \emph{Hartree-Fock theory for double condensates}.
 \newblock Phys. Rev. Lett. \textbf{78}(1997), 3594-3597.


\bibitem{Lin-Wei-CMP:2005}
  \newblock {T.C. Lin and J.C. Wei},
  \newblock \emph{Ground state of $N$ Coupled Nonlinear Schr\"{o}dinger equations in $\mathbb{R}^n, n\le3$}.
  \newblock Commun. Math. Phys. \textbf{255}(2005), 629-653.

\bibitem{Lin-Wei-PDNP:2006}
  \newblock {T.C. Lin and J.C. Wei},
  \newblock \emph{Solitary and self-similar solutions of two-component system of nonlinear Schr\"odinger equations}.
  \newblock Physics D: Nonlinear Phenomena \textbf{220}(2006), 99-115.

\bibitem{Liu-Wang:2008}
  \newblock {Z.L. Liu and Z.-Q. Wang},
  \newblock \emph{Multiple bound states of nonlinear Schr\"{o}dinger systems}.
  \newblock Comm. Math. Phys. \textbf{282} (2008), 721-731.


\bibitem{Liu-Wang:2010}
  \newblock {Z.L. Liu and Z.-Q. Wang},
  \newblock \emph{Ground states and bound states of a nonlinear Schr\"odinger system}.
  \newblock Advanced Nonlinear Studies \textbf{10} (2010), 175-193.

\bibitem{Maia-Montefusco-Pellacci:2006}
  \newblock {L.A. Maia, E. Montefusco and B. Pellacci},
  \newblock \emph{Positive solutions for a weakly coupled nonlinear Schr\"{o}dinger system}.
  \newblock J. Diff. Equ. \textbf{299} (2006), 743-767.

\bibitem{Mawhin-Willem:1989}
  \newblock {J. Mawhin and M. Willem},
  \newblock \emph{Critical point theory and Hamiltonian system}.
  \newblock Spinger-Verlag, New York 1989.

\bibitem{Mitchell-Chen-Shih-Segev:1996}
  \newblock {M. Mitchell, Z. Chen, M. Shih and M. Segev},
  \newblock\emph{Self-trapping of partially spatially incoherent light.}
  \newblock Phys. Rev. Lett. \textbf{77}  (1996), 490–493.

\bibitem{Ruegg-Cavadini-etc:2003}
  \newblock {Ch. R\"uegg, N. Cavadini, A. Furrer, H.-U. G\"udel, K. Kr\"amer, H. Mutka, A. Wildes, K. Habicht and P. Vorderwischu},
  \newblock \emph{Bose-Einstein condensation of the triplet states in the magnetic insulator TlCuCl3}.
  \newblock Nature \textbf{423} (2003), 62-65.


\bibitem{Oruganti-Shi-Shivaji:2002}
  \newblock {S. Oruganti, J.P. Shi and R. Shivaji},
  \newblock\emph{Diffusive logistic equation with constant yield harvesting, I: steady states}.
  \newblock Trans. Amer. Math. Soc. {\bf 354} (2002), 3601-3619.

\bibitem{Sirakov:2007}
  \newblock {B. Sirakov},
  \newblock \emph{Least energy solitary waves for a system of nonlinear Schr\"odinger equations in $\R^n$}.
  \newblock Comm. Math. Phys. \textbf{271} (2007), 199-221.

\bibitem{Tian-Wang:2011}
  \newblock {R.-S. Tian and Z.-Q. Wang},
  \newblock \emph{Multiple solitary wave solutions of nonlinear Schr\"{o}dinger systems.}
  \newblock Top. Meth. Nonlin. Anal. \textbf{37} (2011), 203-223.

\bibitem{Tian-Wang:2013-jan}
  \newblock {R.-S. Tian and Z.-Q. Wang},
  \newblock \emph{Bifurcation results on positive solutions of an indefinite nonlinear elliptic
system.}
  \newblock Disc. Cont. Dyn. Sys. - Series A \textbf{33} (2013), 335-344.

\bibitem{Tian-Wang:2013-feb}
  \newblock {R.-S. Tian and Z.-Q. Wang},
  \newblock \emph{Bifurcation results on positive solutions of an indefinite nonlinear elliptic
system II}
  \newblock Adv. Non. Stud. \textbf{13} (2013), 245-262.


\bibitem{Wei-Weth:2007}
   \newblock {J. Wei and T. Weth},
   \newblock \emph{Nonradial symmetric bound states for a system of two coupled Schr\"{o}dinger equations}.
   \newblock Rend. Lincei Mat. Appl. \textbf{18} (2007), 279-293.

}
\end{thebibliography}
\end{document}